\pdfoutput=1 

\documentclass{article}
\usepackage{excludeonly}

\usepackage[disable]{todonotes}

\usepackage{savesym}
\usepackage{amsthm,amsmath}
\usepackage{amssymb,latexsym,graphicx}
\usepackage{amscd}
 \usepackage[all,cmtip]{xy}
\usepackage{tikz-cd}

\usepackage{accents}
\usepackage{cite}
\usepackage{mathtools}
\usepackage{stmaryrd} 

\usepackage{calligra}
\DeclareMathAlphabet{\mathcalligra}{T1}{calligra}{m}{n}
\DeclareMathAlphabet{\mathpzc}{OT1}{pzc}{m}{it}
\usepackage{hycolor}
\usepackage{xcolor}
\usepackage[
                       colorlinks=true,
                       linkcolor=black, 
                       citecolor=black, 
                       urlcolor=blue,
                     ]{hyperref}
\usepackage{soul} 

\usepackage{makeidx}

\usepackage[intoc]{nomentbl}
\makenomenclature

\usepackage{stackengine} 
\usepackage{booktabs} 

\newtheorem{theoremABC}{Theorem}

\newtheorem{theorem}{Theorem}[section]

\newtheorem{corollary}[theorem]{Corollary}

\newtheorem{lemma}[theorem]{Lemma}
\newtheorem{proposition}[theorem]{Proposition}

\theoremstyle{definition}
\newtheorem{definition}[theorem]{Definition}

\newtheorem{remark}[theorem]{Remark}

\theoremstyle{remark}

\newcommand{\N}{{\mathbb{N}}}

\newcommand{\R}{{\mathbb{R}}}
\renewcommand{\SS}{{\mathbb{S}}}

\newcommand{\Z}{{\mathbb{Z}}}
\newcommand{\Aa}{{\mathcal{A}}}   
\newcommand{\Bb}{{\mathcal{B}}}

\newcommand{\Ff}{{\mathcal{F}}}
\newcommand{\Hh}{{\mathcal{H}}}

\newcommand{\Vv}{{\mathcal{V}}}

\newcommand{\coker}{{\rm coker\, }}  
\newcommand{\SPAN}{{\rm span\, }}         

\newcommand{\Grad}{\mathop{\mathrm{Grad}}}    
\newcommand{\INDEX}{\mathop{\mathrm{index}}}     
\newcommand{\IND}{{\rm ind}}                  

\newcommand{\cgraph}[1]{\Gamma_{\kern-.5ex{}#1}}     
\newcommand{\ind}{{\rm Ind}} 
\newcommand{\Crit}{{\rm Crit}}        
\newcommand{\Ho}{{\rm H}}              
\newcommand{\HM}{{\rm HM}}             
\newcommand{\CM}{{\rm CM}}             

\newcommand{\ev}{{\rm ev}}

\newcommand{\norm}{{\rm norm}}

\newcommand{\orient}[1]{\langle #1 \rangle}  
\newcommand{\inner}[2]{\langle #1, #2\rangle}   
\newcommand{\INNER}[2]{\left\langle #1, #2\right\rangle}

\def\NABLA#1{{\mathop{\nabla\kern-.5ex\lower1ex\hbox{$#1$}}}}
\def\Nabla#1{\nabla\kern-.5ex{}_{#1}}
\def\Tabla#1{\Tilde\nabla\kern-.5ex{}_{#1}}

\def\norm#1{\mathopen\|#1\mathclose\|}
\def\Norm#1{\left\|#1\right\|}

\renewcommand{\Tilde}{\widetilde}

\newcommand{\p}{{\partial}}

\newlength\eqshift
\renewcommand\theequation{\thesection.\arabic{equation}}
\let\savetheequation\theequation

\makeatletter
\renewcommand*\env@matrix[1][\arraystretch]{%
  \edef\arraystretch{#1}%
  \hskip -\arraycolsep
  \let\@ifnextchar\new@ifnextchar
  \array{*\c@MaxMatrixCols c}}
\makeatother

\makeatletter
\let\save@mathaccent\mathaccent
\newcommand*\if@single[3]{%
  \setbox0\hbox{${\mathaccent"0362{#1}}^H$}%
  \setbox2\hbox{${\mathaccent"0362{\kern0pt#1}}^H$}%
  \ifdim\ht0=\ht2 #3\else #2\fi
  }
\newcommand*\rel@kern[1]{\kern#1\dimexpr\macc@kerna}
\newcommand*\widebar[1]{\@ifnextchar^{{\wide@bar{#1}{0}}}{\wide@bar{#1}{1}}}
\newcommand*\wide@bar[2]{\if@single{#1}{\wide@bar@{#1}{#2}{1}}{\wide@bar@{#1}{#2}{2}}}
\newcommand*\wide@bar@[3]{%
  \begingroup
  \def\mathaccent##1##2{%
    \let\mathaccent\save@mathaccent
    \if#32 \let\macc@nucleus\first@char \fi
    \setbox\z@\hbox{$\macc@style{\macc@nucleus}_{}$}%
    \setbox\tw@\hbox{$\macc@style{\macc@nucleus}{}_{}$}%
    \dimen@\wd\tw@
    \advance\dimen@-\wd\z@
    \divide\dimen@ 3
    \@tempdima\wd\tw@
    \advance\@tempdima-\scriptspace
    \divide\@tempdima 10
    \advance\dimen@-\@tempdima
    \ifdim\dimen@>\z@ \dimen@0pt\fi
    \rel@kern{0.6}\kern-\dimen@
    \if#31
      \overline{\rel@kern{-0.6}\kern\dimen@\macc@nucleus\rel@kern{0.4}\kern\dimen@}%
      \advance\dimen@0.4\dimexpr\macc@kerna
      \let\final@kern#2%
      \ifdim\dimen@<\z@ \let\final@kern1\fi
      \if\final@kern1 \kern-\dimen@\fi
    \else
      \overline{\rel@kern{-0.6}\kern\dimen@#1}%
    \fi
  }%
  \macc@depth\@ne
  \let\math@bgroup\@empty \let\math@egroup\macc@set@skewchar
  \mathsurround\z@ \frozen@everymath{\mathgroup\macc@group\relax}%
  \macc@set@skewchar\relax
  \let\mathaccentV\macc@nested@a
  \if#31
    \macc@nested@a\relax111{#1}%
  \else
    \def\gobble@till@marker##1\endmarker{}%
    \futurelet\first@char\gobble@till@marker#1\endmarker
    \ifcat\noexpand\first@char A\else
      \def\first@char{}%
    \fi
    \macc@nested@a\relax111{\first@char}%
  \fi
  \endgroup
}
\makeatother

\long\def\symbolfootnote[#1]#2{\begingroup%
\def\thefootnote{\fnsymbol{footnote}}\footnote[#1]{#2}\endgroup}

\begin{document}
\sloppy

\author{\quad Urs Frauenfelder \quad \qquad\qquad
             Joa Weber\footnote{
  Email: urs.frauenfelder@math.uni-augsburg.de
  \hfill
  joa@ime.unicamp.br
  }
        %
        %
    \\
    Universit\"at Augsburg \qquad\qquad
    UNICAMP
}

\title{The regularized free fall \\ II -- Homology computation via
  heat flow}

\date{\today}

\maketitle 
%


%
%

%





\begin{abstract}
In~\cite{Barutello:2020a} Barutello, Ortega, and Verzini
introduced a non-local functional which regularizes the free fall.
This functional has a critical point at infinity and therefore does
not satisfy the Palais-Smale condition.

In this article we study the $L^2$ gradient flow
which gives rise to a non-local heat flow.
We construct a rich cascade Morse chain complex 
which has one generator in each degree $k\ge 1$.
Calculation reveals a rather poor Morse homology
having just one generator.
In particular, there must be a wealth of solutions
of the heat flow equation. These can be interpreted 
as solutions of the Schr\"odinger equation after a Wick rotation.
\end{abstract}

\tableofcontents

\section{Introduction}

The free fall describes the motion of a particle on a line in the
gravitational field of a heavy body. The particle will after some time
collide with the heavy body.
However, collisions can be regularized so that after collision the
particle bounces back.
An interesting new approach for regularizing collisions was
discovered in the recent paper~\cite{Barutello:2020a}
by Barutello, Ortega, and Verzini.
Change of time gives rise to a \emph{delayed}, that is non-local,
regularized functional $\Bb$ with an intriguing mathematical structure.

In fact, there are two non-local functionals describing the free fall,
namely, a Lagrangian version $\Bb$, defined in~(\ref{eq:Bb})
below, and a Hamiltonian version~$\Aa_\Hh$.
The two functionals are related to each other by a non-local Legendre transform
as studied in~\cite{Frauenfelder:2021a}.
In the present article we compute the Morse homology of the Lagrangian version
$\Bb$ with respect to an $L^2$ metric.
This is a non-local analogue of the heat flow Morse homology
of the second author~\cite{weber:2013b,weber:2013a,weber:2014c}.

The significance of the free fall lies in the fact that it is the
starting point of the exploration of more complicated systems
like the Helium problem which is an active topic of research
of the first named author with Cieliebak and Volkov~\cite{Cieliebak:2021a}.

\smallskip
In this paper we introduce the heat flow homology for the Lagrangian
functional $\Bb$ of the free fall and compute it.
It turns out that there is a rich interplay between critical points and
gradient flow lines. Although the chain groups are infinite
dimensional it turns out that in sharp contrast
the heat flow Morse homology for the free fall is extremely meager, it
is actually concentrated in degree~$1$.
In particular, by the contrast principle
``large chain complex -- low homology'' many solutions of the
heat flow must exist.

To be more precise since the functional $\Bb$ is not Morse, but only
Morse-Bott, one has to modify standard Morse homology.
We shall choose cascade Morse homology,
established by the first author in his PhD thesis, see~\cite{Frauenfelder:2004a},
because it continues to work with the original gradient flow
of $\Bb$; this is not the case if one perturbs $\Bb$
and does Morse homology with a nearby Morse functional $\widetilde\Bb$.

\begin{theoremABC}\label{thm:main}
The cascade Morse homology of the Morse-Bott functional $\Bb$ is
\[
   \HM_*(\Bb;\Z_2)
   =
   \begin{cases}
      \Z_2&\text{, $*=1$}
      \\
      0&\text{, else}
   \end{cases}
\]
\end{theoremABC}

\begin{proof}
Proposition~\ref{prop:cascade-complex}.
\end{proof}

An interesting aspect of the heat flow equation is that after applying
a Wick rotation, that is considering imaginary time, one obtains a
solution of the Schr\"odinger equation.

\medskip
In two planned future articles III and IV we intend to study the
Hamiltonian analogue of the heat flow homology in order to obtain
a non-local Floer homology and relate the two
by an adiabatic limit in the spirit of~\cite{salamon:2006a}.
In the first step of this project, article I, we
proved~\cite{Frauenfelder:2021a} that the Fredholm indices in both
theories agree. The gradient flow equations in the Hamiltonian theory
are non-local perturbed holomorphic curve equations
which after a Wick rotation become solutions of a transport equation
and hence solve a wave equation.

\medskip
Theorem~\ref{thm:main} might also be interpreted that
the Morse homology of the functional $\Bb$ computes the homology
of a Conley pair $(N,L)$ where $N$ is the domain of $\Bb$
and $L:=\{\Bb<C_1\}$ is the sub-level set corresponding to the lowest
critical value.
It is therefore conceivable that Theorem~\ref{thm:main}
can be proved by an infinite dimensional Conley index argument
as in~\cite{weber:2014c}; for a short overview see~\cite{weber:2014b}.

\smallskip 

In the present paper we follow a different approach by arguing directly
with the Morse complex without reference to the topology of the
underlying space.
However, we do not provide a direct existence proof
of the heat flow gradient flow lines. We deduce their existence
with some tricks.
The crucial observation is that if one fixes the asymptotics then
the heat flow gradient flow lines lie in some finite dimensional subspaces.
This allows us to deduce the existence of gradient flow lines
by considering the finite dimensional Morse homology of the
restriction of the action functional to the finite dimensional subspaces.
Now the crucial step is to
chose the auxiliary Morse function on the critical manifold
in such a way that the Morse-Smale condition holds
simultaneously on the full space as well as on the finite dimensional
subspaces.

\medskip
Although triviality of the negative bundles $\Vv^- C_k$ over the critical
manifolds $C_k$ is not used in the present article (since we use $\Z_2$
coefficients only), triviality is relevant for $\Z$ coefficients and this
is why we include the proof as an appendix.
In the appendix we also include the short argument, implicit
in~\cite{Frauenfelder:2021a}, that the functional $\Bb$ is indeed
Morse-Bott of nullity 1.

\bigskip\noindent
{\bf Idea of proof}
The functional $\Bb$ is Morse-Bott and its critical manifold $C$
consists of countably many circles $C_k\cong \SS^1$ of
odd Morse indices $2k-1$ for $k\in\N$, as illustrated
on the left hand side of Figure~\ref{fig:fig-cascade-complex}.

\begin{figure}
  \centering
  \includegraphics
                             {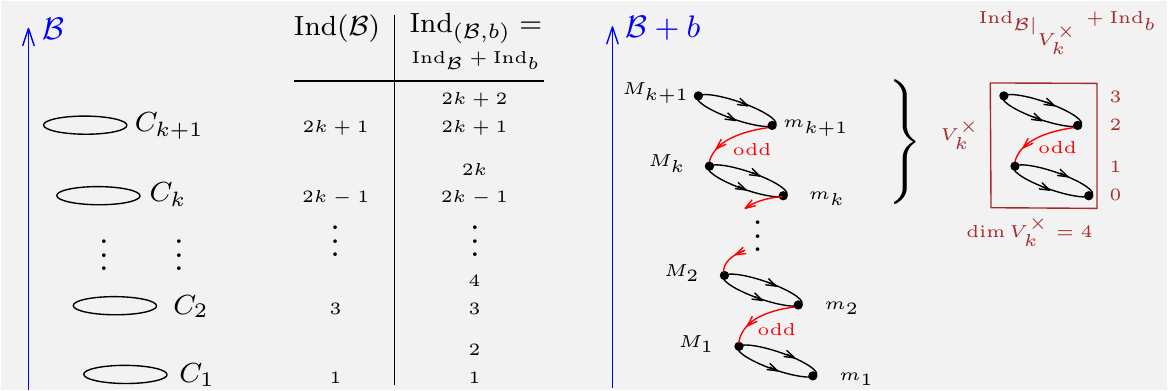}
  \caption{Cascade complex - the cycles are $m_1$ and all $M_k=\p m_{k+1}$}
   \label{fig:fig-cascade-complex}
\end{figure} 

We choose on each of the circles $C_k$
an auxiliary Morse function $b_k$ having exactly one maximum
$M_k$ and exactly one minimum $m_k$,
as illustrated on the right hand side of Figure~\ref{fig:fig-cascade-complex}.
The cascade indices $\ind_{(\Bb,b)}$ are defined and given by
\begin{equation*}
\begin{split}
   \ind_{(\Bb,b)}(M_k)
   &:=\ind_\Bb(M_k)+\ind_{b}(M_k)=2k
   \\
   \ind_{(\Bb,b)}(m_k)
   &:=\ind_\Bb(m_k)+\ind_{b}(m_k)=2k-1
\end{split}
\end{equation*}
Therefore the cascade chain groups have exactly one generator
in each degree
\begin{equation*}
   \CM_\ell(\Bb,b;\Z_2)
   =
   \begin{cases}
     \Z_2 \orient{M_k}&\text{, $\ell=2k$}
     \\
     Z_2 \orient{m_k}&\text{, $\ell=2k-1$}
   \end{cases}
\end{equation*}
whenever $\ell\in\N$ and they are zero else.
On each circle $C_k$ there are two gradient flow lines
from $M_k$ to $m_k$. Since we count gradient flow lines
modulo two we have $\p M_k=0$.
More subtle is to count the cascades from
$m_{k+1}$ to $M_k$. These are solutions of the non-local
heat flow equation. We do not construct them directly, but
deduce their existence indirectly via the following {\bf crucial observation}:
the heat flow gradient flow lines lie in finite dimensional subspaces
$V_k^\times$, in fact $\dim V_k^\times=4$.
The restriction $\Bb_k$ of the functional $\Bb$ to $V_k^\times$
has as critical point set precisely the circles $C_{k+1}$ and $C_k$.
We prove that by careful choice of the auxiliary Morse function $b_k$
on $C_k$ we can achieve that our gradient flow equation
satisfies the Morse-Smale condition simultaneously
as well on $V_k^\times$ as on the full space.
This allows us to consider the cascade complex of $\Bb_k$
as a sub-complex of the cascade complex of $\Bb$
whose degree is however shifted by $2k-1$.
On the finite dimensional subspaces $V_k^\times$
we can use topology to compute the cascade Morse homology
which turns out to be the homology of the 3-dimensional sphere $\SS^3$
which has one generator in degree 0 and one generator in degree 3.
Therefore we can conclude that on the finite dimensional subspaces
$V_k^\times$ there is an {\color{red} odd number} of gradient flow lines from
$m_{k+1}$ to $M_k$.

By our crucial observation the gradient flow lines
of $\Bb$ from $m_{k+1}$ to $M_k$ are precisely the gradient flow lines
of $\Bb_k$ from $m_{k+1}$ to $M_k$.
Thus $\p m_{k+1}=M_k$. Here and throughout we count modulo two.
In particular, the minima $m_{k+1}$ are no cycles, while the maxima
$M_k$ are cycles but boundaries as well. Hence the only cycle which is
not a boundary is the overall minimum $m_1$.
Therefore the homology has a single generator and this generator
sits in degree 1.

\smallskip\noindent
{\bf Acknowledgements.}
UF acknowledges support by DFG grant
FR 2637/2-2.

\section{The Morse-Bott functional \boldmath$\Bb$}
\unboldmath

A quite new approach to the regularization of collisions was
discovered in the recent paper~\cite{Barutello:2020a}
by Barutello, Ortega, and Verzini where the change of time
leads to a \emph{delayed} functional.
In the case of the 1-dimensional Kepler problem
this functional attains the following form
\begin{equation}\label{eq:Bb}
\begin{split}
   \Bb\colon W^{1,2}_\times:=W^{1,2}(\SS^1,\R)\setminus \{0\}&\to\R \\
   q&\mapsto {\color{brown} 4 \norm{q}^2}\,\tfrac12\norm{\dot
   q}^2+\frac{1}{\norm{q}^2}
\end{split}
\end{equation}
where $\norm{\cdot}$ is the $L^2$ norm associated to the $L^2$ inner
product $\inner{\cdot}{\cdot}$.
One might interpret this functional
as a non-local mechanical system consisting of
kinetic minus potential energy.
As shown in~\cite{Frauenfelder:2021a} the differential
\begin{equation*}
\begin{split}
   d\Bb\colon W^{1,2}_\times \times W^{1,2}
   :=W^{1,2}(\SS^1,\R)\setminus \{0\}\times W^{1,2}(\SS^1,\R)&\to\R
\end{split}
\end{equation*}
is given by
\begin{equation*}
\begin{split}
   d\Bb(q,\xi)
   &=4\inner{q}{\xi}\norm{\dot q}^2+4\norm{q}^2\inner{\dot q}{\dot\xi}
   -2\frac{\inner{q}{\xi}}{\norm{q}^4}\\
   &=
   {\color{brown} 4\norm{q}^2}\INNER{-\ddot q+
   \underbrace{\left(
      \frac{\norm{\dot q}^2}{\norm{q}^2}-\frac{1}{2\norm{q}^6}
   \right)}_{=:\alpha} q}
   {\xi}
\end{split}
\end{equation*}
where identity two is valid for sufficiently regular $q$,
say $q\in W^{2,2}(\SS^1,\R)\setminus \{0\}$.

\begin{lemma}[Critical points, {\cite{Frauenfelder:2021a}}]
\label{le:crit-pts-Bb}
The functional $\Bb\colon W^{1,2}_\times\to\R$

The set $\Crit \Bb$ of critical points of $\Bb$ consists of the functions
\begin{equation}\label{eq:q-heat}
   q_k(t)=c_k\cos 2\pi kt,\quad
   c_k=\frac{1}{2^\frac{1}{6}(\pi k)^\frac{1}{3}}
   \in(0,1)
   ,\quad k\in\N
\end{equation}
and their time shifts
\begin{equation}\label{eq:q-shift}
   \left(\sigma_* q_k\right)(t):=q_k(t+\sigma)
   =c_k\bigl(
   \cos 2\pi k\sigma \underbrace{\cos 2\pi k t}_{=:\phi_k(t)}
   -\sin 2\pi k\sigma \underbrace{\sin 2\pi k t}_{=:\psi_k(t)}
   \bigr)
\end{equation}
where $\sigma,t\in\SS^1$.
The corresponding critical values are given by
\begin{equation}\label{eq:crit-values}
   \Bb(q_k)
   =\Bb(\sigma_*q_k)
   =2^\frac{1}{3} 3(\pi k)^\frac{2}{3}
\end{equation}
and the Morse indices are
\begin{equation}\label{eq:ind}
   \ind(q_k)
   =\ind(\sigma_*q_k)
   =2k-1
\end{equation}
\end{lemma}

\boldmath
\section{\boldmath$L^2_q$ gradient equation and flow lines}
\unboldmath

We consider the following metric on $W^{1,2}_\times$.
Given a point $q\in W^{1,2}_\times$ and two tangent vectors
\[
\xi_1,\xi_2\in T_q W^{1,2}_\times=W^{1,2}(\SS^1,\R)=:W^{1,2}
\]
we define what we call the \textbf{\boldmath$L^2_q$ inner product} by
\begin{equation}\label{eq:q-metric}
   \inner{\xi_1}{\xi_2}_q
   :={\color{brown} 4\norm{q}^2} \inner{\xi_1}{\xi_2}
   ,\qquad
 \text{where $\inner{\xi_1}{\xi_2}:=\int_0^1 \xi_1(t)\xi_2(t)\, dt$}
\end{equation}
Note that $\inner{\cdot}{\cdot}$ is the standard $L^2$ inner product
on $L^2(\SS^1,\R)$. In this notation
\begin{equation*}
   \Bb(q)
   =\tfrac12\INNER{\dot q}{\dot q}_q+\frac{1}{\norm{q}^2}
   ,\qquad
   d\Bb(q,\xi)
   =
   \INNER{-\ddot q+\alpha q}
   {\xi}_q
\end{equation*}
The $L^2_q$ gradient of $\Bb$ at $q$ is denoted and given by
\begin{equation}\label{eq:grad-Bb}
   \Grad\Bb(q)
   =-\ddot q+\alpha q,\qquad
   \alpha=\alpha_q :=\left(
   \frac{\norm{\dot q}^2}{\norm{q}^2}-\frac{1}{2\norm{q}^6}
   \right) \in\R
\end{equation}

\boldmath
\subsection*{Flow lines}
\unboldmath

A smooth cylinder $u\colon\R\times\SS^1\to\R$ whose associated path
of loops $s\mapsto u_s:=u(s,\cdot)$ avoids the zero loop
is called a \textbf{heat flow line}\footnote{
  a downward gradient flow line in the $L^2_q$ metric
  of the non-local action functional $\Bb$
  }
if it satisfies the scale ode given by
\begin{equation}\label{eq:heat}
   \Ff(u):=\p_s u -\p_t\p_t u+\alpha_s u =0,\qquad
   \alpha_s:=
      \frac{\norm{\p_t u_s}^2}{\norm{u_s}^2}-\frac{1}{2\norm{u_s}^6}
\end{equation}

\begin{remark}[Wick rotation]
If one considers the above heat flow equation~(\ref{eq:heat})
in imaginary time $is$, corresponding to a Wick rotation,
one obtains the following non-local Schr\"odinger equation
\begin{equation*}\label{eq:Schroedinger}
   i\p_s u -\p_t\p_t u+\alpha_s u =0
\end{equation*}
\end{remark}

\begin{remark}[Asymptotic boundary values of heat flow lines]
If a heat flow line $u$, that is any smooth cylinder
$u\colon\R\times\SS^1\to\R$ such that $\Ff(u)=0$,
admits a non-empty $\omega$-limit set,
then this set $\omega_\pm(u)=\{q_\pm\}$ consists of a single critical
point~(\ref{eq:q-heat}) of the functional $\Bb$
(this holds since $\Bb$ is Morse-Bott by Lemma~\ref{le:Morse-Bott}).
In this case it is well known
that the gradient flow line $s\mapsto u_s:=u(s,\cdot)$
converges exponentially to $q_\pm$, as flow time $s\to\pm\infty$.
The exponential rate of decay is determined by the spectral gap,
namely, the smallest absolute value of a non-zero eigenvalue.

In the case of the functional $\Bb$, non-emptiness of the $\omega$-limit set
$\omega_\pm(u)$ is not guaranteed, neither in the forward direction
by trying to exploit the facts that there is a forward semi-flow and $\Bb$ is
bounded below (unfortunately a minimum is not achieved due to
escape to infinity), nor in both directions by imposing a finite
energy condition on $u$.
\end{remark}

\boldmath
\subsection*{Linearization}
\unboldmath

We shall linearize the map $\Ff$ defined by~(\ref{eq:heat})
at any smooth cylinder $u\colon\R\times\SS^1\to\R$ which has as
asymptotic boundary conditions two critical points,
see~(\ref{eq:q-shift}), of the Morse-Bott functional $\Bb$, in symbols
\begin{equation}\label{eq:asymp-lim}
   q_\pm:=\lim_{s\to\pm\infty} u(s,\cdot)\in Crit\,\Bb
\end{equation}

\begin{definition}
\label{eq:path-spaces}
Suppose $H$ is a separable Hilbert space.
Fix a monotone cutoff function $\beta\in C^\infty(\R,[-1,1])$
with $\beta(s)=-1$ for $s\le-1$ and $\beta(s)=1$ for $s\ge 1$.
Fix a constant $\delta\in (0,4\pi^2)$ and,\footnote{
  The interval $(0,4\pi^2)$ is contained in the spectral gap
  of any Hessian operator $A_{\sigma_* q_k}$ associated to a critical
  point; see~(\ref{eq:eigenvalues}).
  }
see Figure~\ref{fig:fig-Morse-exp-weight},
define a function $\gamma_\delta:\R\to\R$ by
\[
     \gamma_\delta(s):=e^{\delta\beta(s) s}.
\]
Pick a constant $p\in(1,\infty)$.
\begin{figure}
  \centering
  \includegraphics
                             [height=4cm]
                             {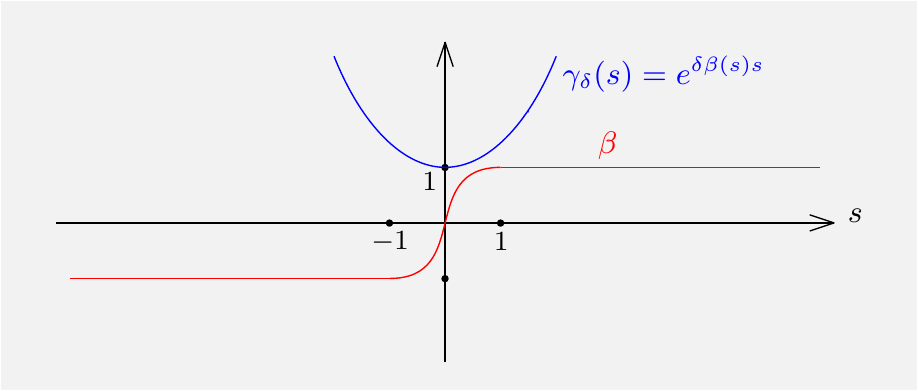}
  \caption{Monotone cutoff function $\beta$ and exponential weight $\gamma_\delta$}
  \label{fig:fig-Morse-exp-weight}
\end{figure}
Consider the Hilbert space valued Sobolev spaces defined for $k\in\N_0$ by
\begin{equation}\label{eq:W-kp-delta}
     W^{k,p}_\delta(\R,H)
     :=\{v\in W^{k,p}(\R,H)\mid
     \gamma_\delta v\in W^{k,p}(\R,H)\}
\end{equation}
with norm
$
     \Norm{v}_{W^{k,p}_\delta}
     :=\Norm{\gamma_\delta v}_{W^{k,p}}
$.
These spaces are Banach space, see e.g.~\cite[App.~A.2]{Frauenfelder:2021b}.
\end{definition}

Given a smooth cylinder $u\colon \R\times\SS^1\to\R$ subject to
asymptotic boundary conditions~(\ref{eq:asymp-lim}) and a smooth
compactly supported function $\xi\colon \R\times\SS^1\to\R$,
pick a family $u^\tau$ such that $u^0=u$ and
$\left.\tfrac{d}{d\tau}\right|_{\tau=0}u^\tau=\xi$,
say $u+\tau\xi$. Abbreviating $W^{k,2}=W^{k,2}(\SS^1)$,
then the linearization
\begin{equation}\label{eq:D_u-dom}
   D_u:=D\Ff(u)
   \colon W^{0,2}_\delta(\R,W^{2,2})\cap W^{1,2}_\delta(\R,W^{1,2})
   \to
   W^{0,2}_\delta(\R,L^2)
\end{equation}
is of the form
\begin{equation*}
\begin{split}
   D_u\xi:
   &=D\Ff(u)\xi :=\left.\tfrac{d}{d\tau}\right|_{\tau=0}\Ff(u^\tau)
   \\
   &=\left.\tfrac{d}{d\tau}\right|_{\tau=0}
   \left(
      \p_s u^\tau -\p_t^2 u^\tau+\alpha_{s,\tau} u^\tau
   \right)
   \\
   &=\p_s\xi-\p_t\p_t\xi+\alpha_s \xi
   +\left(\left.\tfrac{d}{d\tau}\right|_{\tau=0}\alpha_{s,\tau}\right) u
   \\
\end{split}
\end{equation*}
Further calculation 
shows that at any smooth cylinder $u$ we obtain
\begin{equation}\label{eq:D_u-xi}
\begin{split}
   D_u\xi
   &=\p_s\xi-\p_t\p_t\xi+\alpha_s \xi
   \\
   &\quad
   -2\frac{\inner{\p_t\p_tu_s}{\xi_s}}{\norm{u_s}^2}\, u
   -2\left(\frac{\norm{\p_tu_s}^2}{\norm{u_s}^4}
   -\frac{3/2}{\norm{u_s}^8}\right) \inner{u_s}{\xi_s}\, u
   \\
   &=\p_s\xi-\p_t\p_t\xi+\alpha_s \xi
   \\
   &\quad
   -\frac{2}{\norm{u_s}^2}\Biggl(
   \langle\underbrace{\p_t\p_tu_s}_{\p_su_s+\alpha_su_s} , \xi_s\rangle
   +\biggl(\underbrace{
     \frac{\norm{\p_tu_s}^2}{\norm{u_s}^2}
     -\frac{3/2}{\norm{u_s}^6}
     }_{\alpha_s-\frac{1}{\norm{u_s}^6}}
   \biggr)
   \inner{u_s}{\xi_s}
   \Biggr)\, u
   \\
   &=\p_s\xi-\p_t\p_t\xi+\alpha_s \xi
   \\
   &\quad
   -\frac{2}{\norm{u_s}^2}
   \left(
   \INNER{\p_s u_s}{\xi_s}+\left(2\alpha_s-\frac{1}{\norm{u_s}^6}\right)
   \INNER{u_s}{\xi_s}
   \right) u
\end{split}
\end{equation}
where the last identity holds whenever $u$ solves the heat equation.

\boldmath
\subsection*{The adjoint of the linearization}
\unboldmath

Given a smooth cylinder $u\colon\R\times\SS^1\to\R$, consider the
\textbf{\boldmath $L^2_u$ inner product} defined by
\[
   \INNER{\xi}{\eta}_u
   :=\int_{-\infty}^\infty\INNER{\xi_s}{\eta_s}_{u_s} ds
   :=\int_{-\infty}^\infty{\color{brown} 4\norm{u_s}^2}
   \INNER{\xi_s}{\eta_s} ds
\]
for compactly supported smooth functions
$\xi,\eta\colon \R\times\SS^1\to\R$.

The $L^2_u$ adjoint operator of $D_u$, notation $D_u^*$, is determined by the identity
\[
   \INNER{D_u\xi}{\eta}_u=\INNER{\xi}{D_u^*\eta}_u
\]
for compactly supported smooth vector fields $\xi$ and $\eta$ along
the cylinder $u$.
To get a formula for $D_u^*$ we rewrite the inner product as follows.
In the first step we use for $D_u\xi$ the equality~(\ref{eq:D_u-xi})
and in the second step we apply partial integration
with respect to $s$ to obtain
\begin{equation*}
\begin{split}
   &\INNER{D_u\xi}{\eta}_u\\
   &=\int_{-\infty}^\infty {\color{brown} 4\norm{u_s}^2}
   \cdot \Bigl(
     \INNER{\p_s\xi_s}{\eta_s}
     -\INNER{\p_t\p_t\xi_s}{\eta_s}
     +\INNER{\alpha_s\xi_s}{\eta_s}
   \Bigr)\, ds
   \\
   &\quad
   +\int_{-\infty}^\infty {\color{brown} 4 \norm{u_s}^2}
   \left({\color{cyan} -}\, 2 \frac{\inner{{\color{cyan}\p_t}\p_tu_s}{\xi_s}}{\norm{u_s}^2}
   -2\frac{\norm{\p_tu_s}^2}{\norm{u_s}^4} \inner{u_s}{\xi_s} 
   +\frac{3}{\norm{u_s}^8} \inner{u_s}{\xi_s}
   \right) \INNER{u_s}{\eta_s} ds
   \\
   &=\int_{-\infty}^\infty\Bigl( 
   - {\color{brown} 8}
   \INNER{u_s}{\p_su_s}\INNER{\xi_s}{\eta_s}
   +{\color{brown} 4\norm{u_s}^2}
     \INNER{\xi_s}{-\p_s\eta_s-\p_t\p_t\eta_s+\alpha_s\eta_s}
   \Bigr)\, ds
   \\
   &\quad
   -2\int_{-\infty}^\infty 
   \left( \frac{\inner{\p_t\p_tu_s}{\xi_s}_{\color{red} u_s}}{\norm{u_s}^2}
   +\frac{\norm{\p_tu_s}^2}{\norm{u_s}^4} \inner{u_s}{\xi_s}_{\color{red} u_s}
   -\frac{3/2}{\norm{u_s}^8} \inner{u_s}{\xi_s}_{\color{red} u_s}
   \right) \INNER{u_s}{\eta_s} ds
   \\
   &=\INNER{\xi}{-\p_s\eta-\p_t\p_t\eta-\alpha\eta}_{\color{red} u}
   -2\int_{-\infty}^\infty\frac{\INNER{u_s}{\p_su_s}}{\norm{u_s}^2}
   \INNER{\xi_s}{\eta_s}_{\color{red} u_s}\, ds
   \\
   &\quad
   -2\int_{-\infty}^\infty
   \left(
  \INNER{\xi_s}{\p_t\p_tu_s}_{\color{red} u_s}
   \frac{\inner{u_s}{\eta_s}}{\norm{u_s}^2} 
   +\INNER{\xi_s}{u_s}_{\color{red} u_s}\left(
   \frac{\norm{\p_tu_s}^2}{\norm{u_s}^4} 
   -\frac{3/2}{\norm{u_s}^8} 
   \right) \inner{u_s}{\eta_s}
   \right) ds
   \\
   &=\INNER{\xi}{D_u^*\eta}_u
\end{split}
\end{equation*}
Hence the $L^2_u$ adjoint of the
linearization $D_u$ is of the form
\begin{equation}\label{eq:D_u-xi*}
\begin{split}
   D_u^*\eta
   &=-\p_s\eta -\p_t\p_t\eta+\alpha_s\eta 
   {\color{orange} \,\,-\,2\frac{\inner{u_s}{\p_su_s}}{\norm{u_s}^2}\,\eta}
   \\
   &\quad
   -2 \frac{\inner{u_s}{\eta_s}}{\norm{u_s}^2}\, \p_t\p_tu
   -2\left(\frac{\norm{\p_tu_s}^2}{\norm{u_s}^4}
   -\frac{3/2}{\norm{u_s}^8}\right) \INNER{u_s}{\eta_s}
    u
\end{split}
\end{equation}
where the yellow extra term arose when we integrated by parts the $s$
variable.

\boldmath
\section{Fourier mode intervals and isolating neighborhoods}
\unboldmath

\boldmath
\subsection*{Flow lines}
\unboldmath

We write $u(s,t)$ for any fixed time $s\in\R$ as a Fourier series in
the form
\begin{equation}\label{eq:Fourier-u_s}
   u_s(t):=u(s,t)
   =a_0(s)+\sum_{k=1}^\infty\left(a_k(s)\cos2\pi kt+b_k(s)\sin2\pi kt\right)
\end{equation}

\begin{proposition}[Isolating neighborhood -- Fourier mode interval]
\label{prop:interval}
Assume that $u$ is a solution of the delayed heat equation~(\ref{eq:heat})
with asymptotic boundary conditions~(\ref{eq:asymp-lim}), that is
\begin{equation}\label{eq:asymp-limit}
   q_\pm(t):=\lim_{s\to\pm\infty} u(s,t)
   =a_{k_\pm}\cos2\pi k_\pm t+b_{k_\pm}\sin2\pi k_\pm t
\end{equation}
for every $t\in\SS^1$ and for some positive integers $k_\pm\in\N$
and constants $a_{k_\pm}$ and $b_{k_\pm}$; cf.~(\ref{eq:q-shift}).
Then the Fourier coefficients $a_k(s)\equiv 0$ and $b_k(s)\equiv 0$
vanish identically for all $k$ outside the interval $[k_+,k_-]$.
\end{proposition}

\begin{proof}
Pick a Fourier mode $k\in\N_0$.
With the constants defined by
\begin{equation}\label{eq:a_k^pm}
   a_k^\pm:=\begin{cases}
   0&\text{, $k\not=k_\pm$}\\
   a_{k_\pm}&\text{, $k=k_\pm$}\\
   \end{cases}
   ,\qquad
   b_k^\pm:=\begin{cases}
   0&\text{, $k\not=k_\pm$}\\
   b_{k_\pm}&\text{, $k=k_\pm$}\\
   \end{cases}
\end{equation}
we obtain the identity $\lim_{s\to\pm\infty} a_k(s)=a_k^\pm$.
Taking one $s$ derivative and two $t$ derivatives
of the Fourier series~(\ref{eq:Fourier-u_s}) the heat equation~(\ref{eq:heat})
implies that
\begin{equation}\label{ref:ode}
   a_k^\prime(s)+\left((2\pi k)^2 +\alpha_s\right) a_k(s)=0
\end{equation}
for every $s\in\R$. Being a first order ode we conclude that
\[
   a_k(0)\not= 0\qquad\Rightarrow\qquad a_k(s)\not=0 \;\;\;\forall s\in\R
\]
So we assume that $a_k(0)\not= 0$. It is useful to calculate the derivative

\[
   \frac{d}{ds} \ln \left(a_k(s)^2\right)
   =\frac{2a_k(s)a_k^\prime(s)}{a_k(s)^2}
   =2 \frac{a_k^\prime(s)}{a_k(s)}
   =-2 (2\pi k)^2-2\alpha_s
\]
where in the last equality we used the ode~(\ref{ref:ode}).

\smallskip
\noindent
{\bf Step 1.} $k<k_+$ $\Rightarrow$ $a_k\equiv 0$
\\
The proof of Step 1 works by showing that the assumption
$a_k(0)\not= 0$ produces a contradiction.
Since $k<k_+$ we get that
\[
   \frac{d}{ds}\Bigl( \ln \left(a_k(s)^2\right)
   -\ln \left(a_{k_+}(s)^2\right)\Bigr)
   =2 (2\pi k)^2 ({k_+}^2-k^2)>0
\]
This shows that for $s>0$ there is the inequality
\[
    \ln \left(a_k(s)^2\right)
   -\ln \left(a_{k_+}(s)^2\right)
   >
   \ln \left(a_k(0)^2\right)
   -\ln \left(a_{k_+}(0)^2\right)
\]
or equivalently
\[
   \ln \left(a_k(s)^2\right)- \ln \left(a_k(0)^2\right)
   > \ln \left(a_{k_+}(s)^2\right)- \ln \left(a_{k_+}(0)^2\right)
\]
for every $s>0$. Exponentiating we get that
\[
   \frac{a_k(s)^2}{a_k(0)^2}
   >\frac{a_{k_+}(s)^2}{a_{k_+}(0)^2}
\]
Taking the limit, as $s\to\infty$, of the right hand side we obtain
\[
   \lim_{s\to\infty}\frac{a_{k_+}(s)^2}{a_{k_+}(0)^2}
   =\frac{a_{k_+}^2}{a_{k_+}(0)^2}
   >0
\]
since $\lim_{s\to\infty} u_s=q_+$.
On the other hand, taking the limit, as $s\to\infty$, of the left hand
side we obtain
\[
   \lim_{s\to\infty}\frac{a_{k}(s)^2}{a_{k}(0)^2}
   =\frac{{\color{cyan} 0}}{a_{k}(0)^2}
   =0
\]
Here we used that the Fourier coefficient for $k$ in $q_+$
{\color{cyan} vanishes}.
The last three displayed formulas contradict each other. Therefore
the assumption that $a_k(0)\not=0$ had to be wrong.
We conclude that $a_k\equiv 0$ vanishes identically if $k<k_+$.

\smallskip
\noindent
{\bf Step 2.} $k>k_-$ $\Rightarrow$ $a_k\equiv 0$
\\
To prove this note that since $k>k_-$ we get that
\[
   \frac{d}{ds}\Bigl( \ln \left(a_k(s)^2\right)
   -\ln \left(a_{k_-}(s)^2\right)\Bigr)
   =2 (2\pi k)^2 ({k_-}^2-k^2)<0
\]
This shows that for $s<0$ there is the inequality
\[
    \ln \left(a_k(s)^2\right)
   -\ln \left(a_{k_-}(s)^2\right)
   >
   \ln \left(a_k(0)^2\right)
   -\ln \left(a_{k_-}(0)^2\right)
\]
or equivalently
\[
   \ln \left(a_k(s)^2\right)- \ln \left(a_k(0)^2\right)
   > \ln \left(a_{k_-}(s)^2\right)- \ln \left(a_{k_-}(0)^2\right)
\]
for every $s>0$. Exponentiating we get that
\[
   \frac{a_k(s)^2}{a_k(0)^2}
   >\frac{a_{k_-}(s)^2}{a_{k_-}(0)^2}
\]
Taking the limit as $s\to-\infty$ of the right hand side we obtain
\[
   \lim_{s\to-\infty}\frac{a_{k_-}(s)^2}{a_{k_-}(0)^2}
   =\frac{a_{k_-}^2}{a_{k_-}(0)^2}
   >0
\]
since $\lim_{s\to-\infty} u_s=q_-$.
On the other hand, taking the limit as $s\to-\infty$ of the left hand
side we obtain
\[
   \lim_{s\to-\infty}\frac{a_{k}(s)^2}{a_{k}(0)^2}
   =\frac{{\color{cyan} 0}}{a_{k}(0)^2}
   =0
\]
Here we used that the Fourier coefficient for $k$ in $q_-$ {\color{cyan} vanishes}.
The last three displayed formulas contradict each other. Therefore
the assumption that $a_k(0)\not=0$ had to be wrong.
We conclude that $a_k(s)\equiv 0$ vanishes identically if $k<k_-$.
The proof for $b_k(s)$ is analogous.
\end{proof}

\boldmath
\subsection*{Linearization}
\unboldmath

We write $u(s,t)$ for any fixed time $s\in\R$ as a Fourier series in
the form of equation~(\ref{eq:Fourier-u_s}).
Similarly we write $\xi\colon\R\times\SS^1\to\R$ for any fixed time
$s\in\R$ as a Fourier series in the form 
\begin{equation}\label{eq:Fourier-xi_s}
   \xi_s(t):=\xi(s,t)
   =A_0(s)+\sum_{k=1}^\infty\left(A_k(s)\cos2\pi kt+B_k(s)\sin2\pi kt\right)
\end{equation}

\begin{proposition}[The kernel of $D_u$ has the same Fourier mode interval as $u$]
\label{prop:interval-eta}
Let $u$ be a solution of the delayed heat equation~(\ref{eq:heat}) with asymptotic
boundary conditions~(\ref{eq:asymp-limit}),
namely, two critical points
\[
   q_\pm:=\lim_{s\to\pm\infty} u(s,\cdot)\in\Crit\,\Bb
\]
where $q_\pm$ is determined by a positive integer $k_\pm\in\N$
and two constants $a_{k_\pm},b_{k_\pm}$. Suppose that $\xi$ is an
element of the kernel of $D_u$, that is $D_u\xi=0$. For
$s\in\R$ write $\xi_s:=\xi(s,\cdot)\colon\SS^1\to\R$ in the form of
the Fourier series~(\ref{eq:Fourier-xi_s}).
Then $A_k(s)\equiv 0$ and $B_k(s)\equiv 0$ vanish identically for all
$k$ outside the interval $[k_+,k_-]$.
\end{proposition}

\begin{proof}
Pick a common Fourier mode $k\in\N_0$ of $u$ and $\xi$. Consider the
constants $a_k^\pm$ and $b_k^\pm$ defined by~(\ref{eq:a_k^pm})
and let $A_k^\pm$ and $B_k^\pm$ be defined analogously.
Taking one $s$ derivative and two $t$ derivatives of the Fourier
series~(\ref{eq:Fourier-xi_s}) the equation $D_u\xi=0$,
see~(\ref{eq:D_u-xi}), and the heat equation~(\ref{eq:heat}) for $u$
provide the ode
\begin{equation}\label{ref:ode-xi}
   A_k^\prime(s)+\left((2\pi k)^2 +\alpha_s\right) A_k(s)
   - (\Phi^*\xi_s)\cdot  a_k(s)
   =0
\end{equation}
for the function $A_k(s)$. Here the function $a_k(s)$ satisfies the
ode~(\ref{ref:ode}) and
\[
   \Phi^*\xi_s
   :=-\frac{2}{\norm{u_s}^2}
   \left(\INNER{\p_t\p_tu_s}{\xi_s}+\left(\alpha_s-\frac{1}{\norm{u_s}^6}\right)
   \INNER{u_s}{\xi_s}\right)
\]
Once we recall that for $k$ outside the interval $[k_+,k_-]$ the
functions $a_k\equiv0$ and $b_k\equiv0$ vanish identically,
the proof of the present proposition reduces to the one of
Proposition~\ref{prop:interval}. Indeed for $k\notin [k_+,k_-]$ the
ode~(\ref{ref:ode-xi}) reduces to the ode
\[
   A_k^\prime(s)+\left((2\pi k)^2 +\alpha_s\right) A_k(s)
   =0
\]
for $A_k(s)$. But this is exactly the ode~(\ref{ref:ode}) for which we already
showed the assertion.
The proof for $B_k(s)$ is analogous.
\end{proof}

\boldmath
\section{Restriction to 4-dimensional subspaces \boldmath$V_k$}
\unboldmath

Fix $k\in\N$ and define functions
\[
    \phi_k(t):=\cos 2\pi kt,\qquad
    \psi_k(t):=\sin 2\pi kt
\]
Consider the 4-dimensional vector subspace 
of the free loop space $W^{1,2}(\SS^1,\R)$
spanned by the following four functions (cf. Lemma~\ref{le:crit-pts-Bb})
\[
   V_k=\SPAN\{\phi_k,\psi_k,\phi_{k+1},\psi_{k+1}\},\qquad
   V_k^\times:=V_k\setminus\{0\}
\]

The following corollary tells that flow lines from $C_{k+1}$ to $C_k$
critical points lie in one and the same $V_k$.

\begin{figure}[h]
  \centering
  \includegraphics
                             {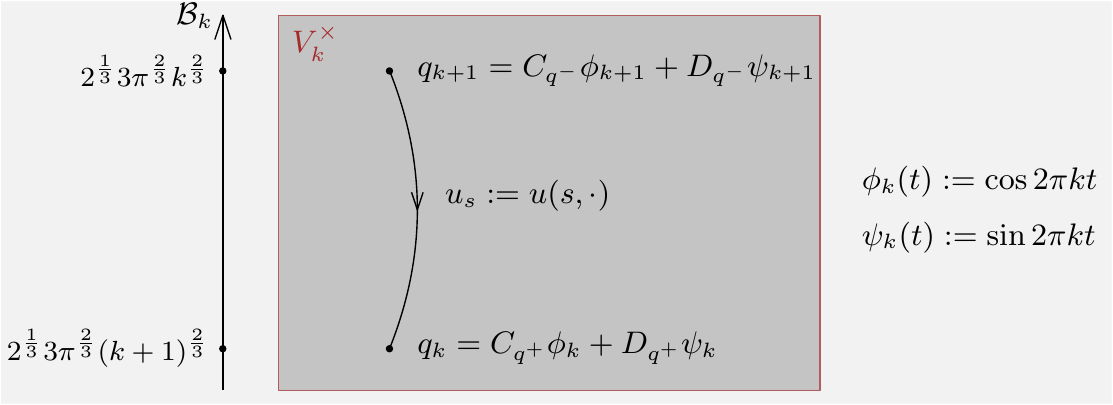}
  \caption{Flow lines connecting consecutive critical manifold components
                 $C_{k+1}$ and $C_k$ lie in a 4-dimensional space $V_k$}
   \label{fig:fig-gradline}
\end{figure} 

\begin{corollary}\label{cor:isolating}
Suppose $u\colon\R\times\SS^1\to\R$ is a gradient flow line
of $\Grad \Bb$, see~(\ref{eq:grad-Bb}), which asymptotically converges
to critical points lying in $V_k$.
Then the whole gradient flow line
$u_s:=u(s,\cdot)$ lies in $V_k^\times$ for all $s\in\R$.
\end{corollary}

\begin{proof}
Proposition~\ref{prop:interval}.
\end{proof}

In view of the above corollary we want to study in detail the
restriction of the functional
\[
   \Bb(q)
   ={\color{brown} 4 \norm{q}^2}\,\tfrac12\norm{\dot
   q}^2+\frac{1}{\norm{q}^2}
\]
to the pointed 4-dimensional subspace $V_k^\times$, notation
\[
   \Bb_k:=\Bb|_{V_k^\times}\colon V_k^\times\to(0,\infty)
\]

\begin{lemma}[Morse indices of the restricted functional]
\label{le:MI-B_k}
For $k\in\N$ it holds that $\Crit\Bb|_{V_k}\subset\Crit\,\Bb$ and that
\[
   \ind_{\Bb|_{V_k}}(q_{k+1})=2
   ,\qquad
   \ind_{\Bb|_{V_k}}(q_k)=0
\]
\end{lemma}

\begin{proof}
This follows from the computation of
the eigenvalues and eigenvectors in~(\ref{eq:coeffs}).
\end{proof}

Note that there are no constant functions on $V_k^\times$ and
therefore $\norm{\dot q}^2$ is bounded away from zero.
Therefore the restriction of $\Bb$ to $V_k^\times$
goes to infinity when $\norm{q}$ moves to infinity or to zero.
In particular, the restriction of $\Bb$ to $V_k^\times$
is a coercive function (pre-images of compacta are compact).
Therefore Morse homology of the coercive functional $\Bb$
represents singular homology of the domain $V_k^\times$ of $\Bb$.
But $V_k^\times$ is homotopy equivalent to the $3$-sphere $\SS^3$.
We summarize these findings in

\begin{lemma}[Morse complex of the restriction $\Bb_k\colon V_k\to\R$]
\label{le:S3}
For $k\in\N$ it holds
\[
   \HM_*(\Bb_k;\Z_2)
   \simeq\Ho_*(V_k^\times;\Z_2)
   \simeq\Ho_*(\SS^3;\Z_2)
   =
   \begin{cases}
      \Z_2&\text{, $*=0,3$}
      \\
      0&\text{, else}
   \end{cases}
\]
\end{lemma}

\boldmath
\section{Construction of a cascade Morse complex for~\boldmath$\Bb$}\label{sec:cascade}
\unboldmath

We choose on each critical manifold $C_k$ a point
$m_k$, that is for each $k\in\N$.
We consider the unstable manifold 
of $m_{k+1}$ with respect to the restriction $\Bb_k$ of
$\Bb$ to $V_k^\times$, notation $W^u_{-\nabla\Bb_k}(m_{k+1})$.

Since the Morse index of $\Bb_k$ along $C_{k+1}$ is 2,
this unstable manifold is a 2-dimensional sub-manifold of $V_k$.
Since $\Bb_k$ is coercive each point $p\not= m_{k+1}$ of the unstable
manifold $W^u_{-\nabla\Bb_k}(m_{k+1})$ converges under the negative
gradient flow of $\Bb_k$ in positive time to a point $y$ on $C_k$.
Hence we obtain a well defined \textbf{evaluation map} given by
\[
   \ev\colon W^u_{-\nabla\Bb_k}(m_{k+1})\setminus\{m_{k+1}\}\to C_k
   ,\quad
   p\mapsto \lim_{t\to+\infty} \varphi^t_{-\nabla\Bb_k} (p)
\]
We choose a regular value of $\ev$ different from $m_k$,
notation $M_k$.

On each $C_k$ (it is diffeomorphic to $\SS^1$)
we choose a Morse function $b_k$ with exactly two critical points,
namely a maximum at $M_k$ and a minimum at $m_k$.
Let $b$ denote the resulting Morse function
on the set $C=\cup_k C_k$ of critical points of $\Bb$.
Note that the Morse index of the critical points is zero or one, namely
$\ind_b(m_k)=0$ and $\ind_b(M_k)=1$.

Hence in view of~(\ref{eq:ind}) for the \textbf{cascade index}
$\ind_{(\Bb,b)}$ we obtain
\begin{equation}\label{eq:casc-ind-1}
\begin{split}
   \ind_{(\Bb,b)}(M_k)
   &:=\ind_\Bb(M_k)+\ind_{b}(M_k)=2k
   \\
   \ind_{(\Bb,b)}(m_k)
   &:=\ind_\Bb(m_k)+\ind_{b}(m_k)=2k-1
\end{split}
\end{equation}
From $M_k$ to $m_k$ there are 2 gradient flow lines of $b$ and since
we count modulo 2 we have for the Morse boundary operator
\begin{equation}\label{eq:p-M_k}
   \p M_k=0
\end{equation}
It remains to compute $\p m_{k+1}$.
Before we can do that we have to make sure that we have a well defined
count of cascades from $m_{k+1}$ to $M_k$.
\\
Hence we consider a gradient flow line $u$ of $\Bb$ from
$m_{k+1}$ to $M_k$ and we need to show that $D_u$ is surjective.
In view of these specific asymptotic boundary conditions,
we know by Proposition~\ref{prop:interval} that
$s\mapsto u_s:=u(s,\cdot)$ takes values in $V_k$.
We consider the restriction of $D_u$ to $V_k$ as an operator
\[
   D_u|_{V_k}\colon W^{1,2}_\delta(\R,V_k)\to W^{0,2}_\delta(\R,V_k)
\]
It follows by Proposition~\ref{prop:interval-eta}
\begin{equation}\label{eq:ker-D_u-restr}
   \ker D_u=\ker D_u|_{V_k}
\end{equation}
Since $M_k$ was chosen as a regular value of the evaluation map $\ev$
we have
\begin{equation}\label{eq:index-D_u-restr}
   \dim\ker D_u|_{V_k}=\INDEX(D_u|_{V_k})
\end{equation}
Furthermore, it is well known that the Fredholm index of the linearization
is given by the cascade index difference of the asymptotic boundary
conditions and this shows the first and the final identity in the following
\begin{equation}\label{eq:index-2}
\begin{split}
   \INDEX(D_u|_{V_k})
   &=\ind_{(\Bb_k,b)}(m_{k+1})-\ind_{(\Bb_k,b)}(M_k)\\
   &=\ind_{\Bb_k}(m_{k+1})+\ind_{b}(m_{k+1})
      -\left(\ind_{\Bb_k}(M_k)+\ind_{b}(M_k)\right)\\
   &=2+0-(0+1)=1\\
   &=\left(2(k+1)-1\right)-2k\\
   &=\ind_{(\Bb,b)}(m_{k+1})-\ind_{(\Bb,b)}(M_k)\\
   &=\INDEX(D_u)
\end{split}
\end{equation}
Here equality two is by definition of the cascade index,
equality three is by Lemma~\ref{le:MI-B_k}, and the penultimate
equality is by~(\ref{eq:casc-ind-1}).

Summarizing, apply successively the
results~(\ref{eq:ker-D_u-restr}),~(\ref{eq:index-D_u-restr}),
and~(\ref{eq:index-2}) to obtain
\begin{equation*}
\begin{split}
   \dim\ker D_u
   &=\dim \ker D_u|_{V_k}\\
   &=\INDEX(D_u|_{V_k})\\
   &=\INDEX(D_u)\\
   :&=\dim\ker D_u-\dim\coker D_u
\end{split}
\end{equation*}
Hence $\coker D_u$ is trivial and therefore the linearized operator
$D_u$ is surjective.

\section{Proof of the main theorem}\label{sec:proof}

\begin{proposition}[Cascade chain complex]\label{prop:cascade-complex}
The cascade chain groups of the Morse-Bott functional $\Bb$
and with respect to the auxiliary Morse function $b$ on
$C:=\Crit\,\Bb$ carefully chosen in Section~\ref{sec:cascade}
are given by
\begin{equation}\label{eq:CM_ell}
   \CM_\ell(\Bb,b;\Z_2)
   =
   \begin{cases}
     \Z_2 \orient{M_k}&\text{, $\ell=2k$}
     \\
     Z_2 \orient{m_k}&\text{, $\ell=2k-1$}
   \end{cases}
\end{equation}
whenever $\ell\in\N$ and they are zero else. All maxima are cycles
\begin{equation}\label{eq:p-M}
   \p M_k=0,\quad M_k=\p m_{k+1},\quad k\in\N
\end{equation}
but also boundaries. There is exactly one more cycle,
the lowest minimum 
\begin{equation}\label{eq:p-m_1}
   \p m_1=0
\end{equation}
and $m_1$ is not a boundary. Thus $m_1$ generates the Morse homology.
\end{proposition}

\begin{proof}[Proof of Proposition~\ref{prop:cascade-complex}]
Assertion~(\ref{eq:CM_ell}) follows from~(\ref{eq:casc-ind-1}).
The first equation in~(\ref{eq:p-M})
follows from~(\ref{eq:p-M_k}).

It remains to prove the second equation $\p m_{k+1}=M_k$
in~(\ref{eq:p-M}). In order to do that we consider the cascade
complex of the restriction $\Bb_k$ of the functional $\Bb$ to
the 4-dimensional space $V_k^\times=V_k\setminus\{0\}$.
\\
The cascade complex of the pair $(\Bb_k,b)$
has four generators, namely $M_{k+1},m_{k+1},M_k,m_k$.
By Lemma~\ref{le:MI-B_k} it holds that
\begin{equation*}
\begin{split}
   \IND_{(\Bb_k,b)}(M_{k+1}):=\IND_{\Bb_k}(M_{k+1})+\IND_{b}(M_{k+1})
   &=2+1=3
\\
   \IND_{(\Bb_k,b)}(m_{k+1}):=\IND_{\Bb_k}(m_{k+1})+\IND_{b}(m_{k+1})
   &=2+0=2
\\
   \IND_{(\Bb_k,b)}(M_{k}):=\IND_{\Bb_k}(M_{k})+\IND_{b}(M_{k})
   &=0+1=1
\\
   \IND_{(\Bb_k,b)}(m_{k}):=\IND_{\Bb_k}(m_{k})+\IND_{b}(m_{k})
   &=0+0=0
\end{split}
\end{equation*}
According to Lemma~\ref{le:S3}
the cascade homology of $(\Bb_k,b)$ on the 4-dimensional space
$V_k^\times$ vanishes in degrees 1 and 2.
Therefore there has to exist an odd number of gradient flow lines
of the restricted functional $\Bb_k$ from $m_{k+1}$ to $M_k$.
According to Corollary~\ref{cor:isolating}
these are precisely the gradient flow lines of the unrestricted
functional $\Bb$ from $m_{k+1}$ to $M_k$.
Therefore $\p m_{k+1}=M_k$.

Because there are no generators of degree lower than
the degree one of $m_1$, it holds that $\p m_1=0$.
\end{proof}

\appendix

\section{Morse-Bott and trivial negative bundles}\label{sec:TNB}

The connected components of the critical
manifold $C:=\Crit\,\Bb$ consist of circles $C_k$
labelled by $k\in\N$.
In~\cite{Frauenfelder:2021a} we already showed that
the kernel of the Hessian of $\Bb$ at each point of $C_k$
is $1$-dimensional. Since the kernel always
contains the tangent space to the critical manifold which in our case
is of dimension one, the two are equal. But this is the definition of
\textbf{Morse-Bott}.
Thus from~\cite{Frauenfelder:2021a} we know that
$C_k\simeq \SS^1$ is Morse-Bott of index $2k-1$.
Since $C_k$ is Morse-Bott there is the splitting
\[
   T_{C_k} W^{1,2}_\times
   =TC_k\oplus\Vv^- C_k\oplus \Vv^+ C_k
\]
which at each point corresponds to the splitting
in zero/negative/positive eigenspaces of the Hessian.
Note that the rank of $\Vv^- C_k$ corresponds to the Morse index of $C_k$.
The above argument proves the following lemma.

\begin{lemma}[Morse-Bott functional]\label{le:Morse-Bott}
The functional $\Bb$ defined by~(\ref{eq:Bb}) is Morse-Bott
and every critical point is of nullity 1.
\end{lemma}

In this section we show additionally that the negative bundle $\Vv^- C_k$
is trivial for each $C_k$.
This plays an important role in order to compute Conley indices
of the critical components $C_k$.

\begin{lemma}[Trivial negative normal bundles]\label{le:neg-bnorm-bdl}
For each $k\in\N$
the negative normal bundle $\Vv^- C_k$ 
over the Morse-Bott manifold $C_k$
is a) trivial and b) of rank $2k-1$.
\end{lemma}

The proof of this lemma covers the following three pages and ends
after equation~(\ref{eq:ghgy767}).
The proof follows the computation of the Morse index
in our previous paper~\cite{Frauenfelder:2021a}.
The new aspect that the line bundle is trivial is to choose a
global trivialization of the restriction of the tangent bundle of $W^{1,2}_\times$
to $C_k$ which has the property that the eigenvalues and eigenvectors
with respect to this global trivialization are independent of the base
point. This then proves that the negative and the positive normal
bundles are both trivial.

\boldmath
\subsubsection*{The Hessian operator -- with respect to $L^2_q$}
\unboldmath

The Hessian operator $A_q$ of the Lagrange functional $\Bb$
is the derivative of the $L^2_q$ gradient at a critical point~$q$, that is
by~(\ref{eq:grad-Bb}) the derivative of the equation
\begin{equation*}
   0=\Grad\Bb(q)
   =-\ddot q+\alpha q,\qquad
   \alpha=\alpha_q :=\left(
   \frac{\norm{\dot q}^2}{\norm{q}^2}-\frac{1}{2\norm{q}^6}
   \right) \in\R
\end{equation*}
where $q\in W^{1,2}(\SS^1,\R)\setminus \{0\}$.

\begin{lemma}[{\cite{Frauenfelder:2021a}}]
The Hessian operator of $\Bb$ at a critical point $q$ is given by
\begin{equation}\label{eq:Hess-Bb}
\begin{split}
   A_q\colon W^{2,2}(\SS^1,\R)&\to L^2(\SS^1,\R)
   \\
   \xi&\mapsto 
   -\ddot\xi+\alpha \xi
   -\frac{2}{\norm{q}^2}\left(2\alpha-\frac{1}{\norm{q}^6}\right)
   \inner{q}{\xi}  q
\end{split}
\end{equation}
\end{lemma}

By~(\ref{eq:q-shift}) the critical points of the functional $\Bb$ are of the form
\begin{equation}\label{eq:q-shift-2}
   \left(\sigma_* q_k\right)(t):=q_k(t+\sigma)
   =c_k\bigl(
   \cos 2\pi k\sigma \underbrace{\cos 2\pi k t}_{=:\phi_k(t)}
   -\sin 2\pi k\sigma \underbrace{\sin 2\pi k t}_{=:\psi_k(t)}
   \bigr)
\end{equation}
for $k\in\N$ and $\sigma,t\in\SS^1$ and
where $c_k=2^{-\frac{1}{6}}(\pi k)^{-\frac{1}{3}} \in(0,1)$.
From now on we fix a critical point $\sigma_* q_k$, that is
$k\in\N$ and $\sigma\in\SS^1$ are fixed from now on.
Taking two $t$ derivatives we conclude that
\[
   \frac{d^2}{dt^2}(\sigma_* q_k) =-(2\pi k)^2 \sigma_* q_k
\]
Since $\frac{d^2}{dt^2}(\sigma_* q_k) =\alpha\cdot (\sigma_* q_k)$ we obtain
\[
   \alpha=\alpha(\sigma_* q_k)
   =-(2\pi k)^2
   =-\frac{2}{c_k^6}
\]
where the last equality is~(\ref{eq:q-heat}).
The formula of the Hessian operator $A_{\sigma_* q_k}$ involves the $L^2$ norm
of $\sigma_* q_k$ and, in addition, the formula of the non-local Lagrange functional $\Bb$
involves $\norm{\sigma_*\dot q_k}^2$.
Straightforward calculation shows that
\begin{equation*}
   \Norm{\sigma_* q_k}^2
   =\Norm{q_k}^2
   \stackrel{(\ref{eq:q-heat})}{=}\frac{c_k^2}{2}
   =\frac{1}{2^\frac{4}{3} (\pi k)^\frac{2}{3}}
   ,\quad
   \Norm{\tfrac{d}{dt}(\sigma_* q_k)}^2
   =(2\pi k)^2\frac{c_k^2}{2}
   =2^\frac{2}{3} (\pi k)^\frac{4}{3}
\end{equation*}
Thus
\[
   \Bb(\sigma_* q_k)
   =2 \norm{\sigma_* q_k}^2\norm{\tfrac{d}{dt}(\sigma_* q_k)}^2+\frac{1}{\norm{\sigma_* q_k}^2}
   =2^\frac{1}{3} 3(\pi k)^\frac{2}{3}
\]
To calculate the formula of $A_{\sigma_* q_k}$ we write $\xi$ as a Fourier series
\[
   \xi
   =\xi_0+\sum_{n=1\atop n\not= k}^\infty\left(
   \xi_n\cos 2\pi nt+\xi^n\sin 2\pi nt
   \right)
   +\xi_k\cos 2\pi k(t+\sigma)+\xi^k\sin 2\pi k(t+\sigma)
\]
where we shifted the $k^{\rm th}$ modes for the reasons explained next.
That the coefficients $\xi_\cdot$ and $\xi^\cdot$ do in fact not
depend on $\sigma$ we will see right after~(\ref{eq:coeffs}).

\begin{remark}[Global trivialization]
The above (partially shifted) Fourier basis depends on the point
$\sigma_* q_k\in C_k$ of the $k^{\rm th}$ component of the critical manifold $C$.
But note that this $\SS^1$-family of Fourier bases
gives a new global trivialization of the restriction of the tangent
bundle of $W^{1,2}_\times$ to $C_k$, namely
\[
   T_{C_k}W^{1,2}_\times\simeq
   C_k\times W^{1,2}(\SS^1,\R).
\]
\end{remark}

We use the orthogonality relation
\[
   \inner{\sigma_* (\cos 2\pi k\cdot)}{\xi}
   =\inner{\cos 2\pi k(\cdot+\sigma)}{\xi}
   =\frac12\xi_k
\]
to calculate the product
\[
   \inner{\sigma_* q_k}{\xi}
   =\inner{c_k\cos 2\pi k(\cdot+\sigma)}{\xi}
   =\frac12 c_k \xi_k.
\]
Putting everything together we recover for the slightly more general
case $\sigma_* q_k$ the result we derived in~\cite{Frauenfelder:2021a}
for $q_k$, namely

\begin{lemma}[Critical values and Hessian]
The critical points of $\Bb$ are of the form~(\ref{eq:q-shift-2})
and at any such $\sigma_* q_k$ the value of $\Bb$ is
\[
   \Bb(\sigma_* q_k)
   =2^\frac{1}{3} 3(\pi k)^\frac{2}{3}
\]
and the Hessian operator~(\ref{eq:Hess-Bb}) of $\Bb$ is
\begin{equation}\label{eq:A-explicit}
\begin{split}
  A_{\sigma_* q_k}\xi
   =-\ddot \xi-(2\pi k)^2\xi
   + {\color{brown} 12(2\pi k)^2 \xi_k}\cos 2\pi k(\cdot +\sigma)
\end{split}
\end{equation}
for every $\xi\in W^{2,2}(\SS^1,\R)$.
\end{lemma}

\subsubsection*{Eigenvalues and Morse index}

Recall that $k\in\N$ and $\sigma\in\SS^1$ are fixed, that is we consider
the given critical point $\sigma_* q_k$.
For the Hessian $A_{\sigma_* q_k}$ given by~(\ref{eq:A-explicit})
we are looking for solutions of the eigenvalue problem
\[
   A_{\sigma_* q_k}\xi=\mu\xi
\]
for $\mu=\mu(\xi;k,\sigma)\in\R$ and $\xi\in W^{2,2}(\SS^1,\R)\setminus\{0\}$.
Observe that
\begin{equation*}
\begin{split}
   -\ddot\xi(t)
   &=\sum_{n=1\atop n\not= k}^\infty (2\pi n)^2
   \left(
      \xi_n\cos 2\pi nt+\xi^n\sin 2\pi nt
   \right)
   \\
   &\quad
   +(2\pi k)^2\xi_k\cos 2\pi k(t+\sigma)+(2\pi k)^2\xi^k\sin 2\pi k(t+\sigma)
   \\
   -(2\pi k)^2\xi(t)
   &=-(2\pi k)^2\xi_0-(2\pi k)^2\sum_{n=1\atop n\not= k}^\infty
   \left(
      \xi_n\cos 2\pi nt+\xi^n\sin 2\pi nt
   \right)
   \\
   &\quad
   -(2\pi k)^2\xi_k\cos 2\pi k(t+\sigma)-(2\pi k)^2\xi^k\sin 2\pi k(t+\sigma).
\end{split}
\end{equation*}
Comparing coefficients in the eigenvalue equation
$A_{\sigma_* q_k}\xi=\mu\xi$ we obtain eigenvectors (left hand side)
and eigenvalues (right hand side) as follows
\begin{equation}\label{eq:coeffs}
\begin{split}
   \cos 2\pi k (t+\sigma) \quad
   &
   \begin{cases}
      \mu\xi_k={\color{brown} 12(2\pi k)^2 \xi_k}
      &\text{ }
   \end{cases}
   \\
   \sin 2\pi k (t+\sigma) \quad
   &
   \begin{cases}
      \mu\xi^k=0
      &\text{ }
   \end{cases}
   \\
   \cos 2\pi nt \quad
   &
   \begin{cases}
      \mu\xi_n=4\pi^2\left(n^2-k^2\right) \xi_n
      &\text{, $\forall n\in\N_0\setminus\{k\}$}
   \end{cases}
   \\
   \sin 2\pi nt \quad
   &
   \begin{cases}
      \mu\xi^n=4\pi^2\left(n^2-k^2\right) \xi^n
      &\text{, $\forall n\in\N\setminus\{k\}$}
   \end{cases}
\end{split}
\end{equation}
We observe that the right hand sides do not depend on $\sigma$.
Therefore the eigenvalues $\mu$, as well as the coefficients
$\xi_\cdot$ and $\xi^\cdot$ of the $\sigma$-dependent Fourier basis,
are all $\sigma$-independent. Since the $\sigma$-dependent Fourier
basis gives rise to the global trivialization we observe that the
negative and the positive part of the normal bundle are both trivial.
This proves part a) of Lemma~\ref{le:neg-bnorm-bdl}.

For the Hessian $A_{\sigma_* q_k}$ one obtains the same eigenvalues
and multiplicities as we obtained in~\cite{Frauenfelder:2021a} in the
unshifted case $q_k$. Indeed the eigenvalues of the Hessian
$A_{\sigma_* q_k}$ are given by
\begin{equation}\label{eq:eigenvalues}
   \mu_n:=4\pi^2(n^2-k^2),\qquad n\in\N\setminus\{k\}
\end{equation}
and by
\[
   \mu_0:=-4\pi^2k^2,
   \qquad\mu_k:=0,
   \qquad\widehat\mu_k:=12(2\pi k)^2
\]
Moreover, their multiplicity (the dimension of the eigenspace) is given by
\[
   m(\mu_n)=2,\quad n\in\N\setminus\{k\},
   \qquad m(\mu_0)=m(\mu_k)=m(\widehat\mu_k)=1
\]
Observe that the eigenvalue $\widehat \mu_k\not= \mu_n$ is different
from $\mu_n$ for every $n\in\N_0$.
Indeed, suppose by contradiction that $\widehat \mu_k=\mu_n$ for some
$n\in\N_0$, that is
\begin{equation}\label{eq:ghgy767}
   48\pi^2k^2=4\pi^2(n^2-k^2)\quad\Leftrightarrow\quad
   13k^2=n^2
\end{equation}
which contradicts that $n$ is an integer.
This proves part b) of Lemma~\ref{le:neg-bnorm-bdl}.

\newpage
\bibliographystyle{alpha}
\addcontentsline{toc}{section}{References}
\bibliography{$HOME/Dropbox/0-Libraries+app-data/Bibdesk-BibFiles/library_math,$HOME/Dropbox/0-Libraries+app-data/Bibdesk-BibFiles/library_math_2020,$HOME/Dropbox/0-Libraries+app-data/Bibdesk-BibFiles/library_physics}{}

%


\end{document}